\documentclass[11pt,reqno]{amsart}
\usepackage[top=1in, bottom=1in, left=1in, right=1in]{geometry}
\usepackage{graphicx}
\usepackage[dvipsnames]{xcolor}
\usepackage{amssymb, amsthm, mathrsfs, amsmath, amsfonts,dsfont,lineno}
\usepackage{epstopdf}
\usepackage{enumerate, enumitem,amsrefs}
\usepackage{lipsum}
\usepackage{hyperref}

\protected\def\ignorethis#1\endignorethis{}
\let\endignorethis\relax

\usepackage{fancyvrb}

\newtheorem{thm}{Theorem}[section]
\newtheorem{lem}{Lemma}[section]
\newtheorem{prop}{Proposition}[section]
\newtheorem{cor}{Corollary}[section]

\theoremstyle{definition}

\newtheorem{rem}{Remark}[section]

\setenumerate[0]{label=(\roman*)}

\def\R{\mathbb{R}}
\def\N{\mathbb{N}}
\def\Z{\mathbb{Z}}
\def\T{\mathbb{T}^2}
\def\C{\mathbb{C}}
\def\W8{W_{\infty}}
\def\d{\delta}

\def\D{\Delta}
\def\d*{\delta_*}
\def\Em{\mathbb{E}_{\mu_{L,\gamma}}}

\def\w{\omega}

\def\supp{{\rm supp}}
\def\E{ E} 
\def\phi{\varphi}

\def\P{\mathbb{P}}

\begin{document}

\title{Invariant measures for the non-periodic two-dimensional Euler equation}

\author[Ana Bela Cruzeiro and Alexandra Symeonides]{Ana Bela Cruzeiro (1) and Alexandra Symeonides (2)}
 \address{
(1) GFMUL and Dep. de Matem\'atica do Instituto Superior T\'ecnico, Univ. de Lisboa \hfill\break\indent 
		Av. Rovisco Pais, 1049-001 Lisboa, Portugal \hfill\break\indent
			}
 \email{abcruz@math.tecnico.ulisboa.pt}

\address{
 (2)  GFMUL and Dep. de Matem\'atica da Faculdade de Ci\^encias, Univ. de Lisboa\hfill\break\indent 
		Campo Grande, 1749-016 Lisboa, Portugal \hfill\break\indent
			}
 \email{asymeonides@fc.ul.pt}
 
 \maketitle
 
\begin{abstract}\noindent
We construct Gaussian invariant measures for the two-dimensional Euler equation on the plane. We show the existence of solution with initial conditions in the support of the measures, namely $H^\beta_{loc}(\R^2)$ with $\beta<-1$. Uniqueness and continuity of the velocity flow are proved. 
\end{abstract}

\tableofcontents

\section{Introduction}

Euler equation describes the time evolution of an incompressible non-viscous fluid with constant density. This fundamental equation has been and still is intensively studied. Among the numerous references on the Euler equation, we cite the books \cite{AK, MB, MP}. It is known  that solutions do not blow up starting from smooth data with finite kinetic energy (T. Kato (1967) \cite{Kato},  C. Bardos (1972) \cite{Bard} among others). Local existence of smooth solutions dates back from Lichtenstein (1925). In two dimensions, for bounded domains and when the initial vorticity is bounded, existence, uniqueness and global regularity of solutions was shown (V.I. Yudovich, 1963 \cite{Y}); these results were extended, in the framework of weak solutions, to the case where the initial vorticity belongs to $L^p$, with $p>1$ and even for $p=1$, when the vorticity is some finite measure.

A more geometric approach,  identifying the solutions of the Euler equation with  velocities of geodesics in a space of diffeomorphisms of the underlying state space, was initiated by V. Arnold (1966) \cite{Arn}. It allowed to show existence of local solutions in some Sobolev spaces (D. G. Ebin and J. Marsden, 1970 \cite{EM}).

Much less is known about irregular solutions of the Euler equation. This paper is devoted to a class of such solutions.

In statistical approaches to hydrodynamics, discussed in the physics literature on turbulence, one considers the evolution of probability densities instead of  pointwise solutions. A major subject of interest is the search for invariant measures. In particular such measures are important because they can be used to prove the existence and study the properties of Euler flows defined almost-everywhere with respect to them.

In this paper we extend the work \cite{AC} in two dimensions to the non-periodic setting.  We prove the existence of  invariant probability measures for the Euler flow and show the existence of these flows, for all times, living in the support of the invariant measures. Those are spaces of very low regularity, namely Sobolev spaces of negative order.

In Section \ref{2dEuler}, we recall the Euler equations in the periodic setting and we fix the notation. For each parameter $\gamma>0$, we denote by $\mu_{L,\gamma}$ the invariant measure for the two-dimensional Euler flow on $[0,L]^2$. These measures $\mu_{L,\gamma}$ were previously constructed in \cite{AC}. In Section \ref{inv_meas}, we show the weak convergence of $\mu_{L,\gamma}$ to some $\mu_\gamma$ in $H^\beta_{loc}(\R^2)$ for $\beta<1$ when the period $L$ tends to infinity. We follow a similar argument used in \cite{ASSuzz} for the Klein-Gordon equation in dimension one. Here we also show that $H^\beta_{loc}(\R^2)$ for $\beta<1$ is the support of $\mu_\gamma$. Finally, in Section \ref{muFlow} we study the $L^p_{\mu_\gamma}$-regularity of the vector field, $B$, and the existence of a unique and globally defined Euler flow, $U$, under which the probability measures $\mu_\gamma$ are invariant. We proceed as follow: following the arguments presented in \cite{AC} we prove the existence of a globally defined pointwise stochastic flow $\tilde U$ for initial data in the support of $\mu_\gamma$. However, as a by product of the uniqueness, we can conclude that this flow is in fact deterministic and we call it $U$. Indeed the proof of uniqueness relies on a result from \cite{AF}, which, in particular, implies that the laws of the pointwise Euler flows are Dirac masses on the trajectories. We conclude this section by proving the continuity of the flow $U$.

\section{2D Euler equations}\label{2dEuler}

Consider the incompressible non-viscous Euler equations on $\R^2$
\begin{equation}\label{Euler}
\frac{\partial u}{\partial t} + (u \cdot \nabla)u =-\nabla p, \qquad \nabla\cdot u=0
\end{equation}
where $u:\R\times\R^2\to\R^2$ denotes the time dependent velocity field and $p:\R\times\R^2\to \R$ denotes the pressure. The first equation is Newton's second law (the acceleration is proportional to the pressure) and the second equation is the incompressibility condition.

\begin{thm}
The time dependent vector field $u$ is a smooth solution of \eqref{Euler} if and only if there exists a smooth (real) function $\varphi$ (stream function) such that $u=\nabla^\perp \varphi$ and $\varphi$ is a solution of the equation
\begin{equation}\label{vorEuler}
\frac{\partial \D\varphi}{\partial t} = -(\nabla^\perp\varphi\cdot \nabla)\D\varphi.
\end{equation}
\end{thm}
\begin{proof}
We refer to \cite{AHKD}.
\end{proof}

Here $\nabla^\perp \phi=(-\partial_2 \phi, \partial_1 \phi)$ where $\partial_1,\partial_2$ denote respectively the partial derivative with respect to the first and second variable. The two problems, \eqref{Euler} and \eqref{vorEuler}, are equivalent; below we consider \eqref{vorEuler}.


\subsection{Periodic case}

We recall here the most relevant results from \cite{AC} about the periodic case. On the space $\T\times\R$, where $\T\simeq [0,L]^2$ such that $L>0$ denotes the period, consider equation \eqref{vorEuler} with periodic boundary condition
$$
\varphi(0,y,t)=\varphi(L,y,t) \mbox{ and } \varphi(x,0,t)=\varphi(x,L,t), \qquad \forall ~(x,y)\in\T.
$$
In \cite{AC} is considered the case $L = 2\pi$, but the analysis for general $L > 0$ is identical if we simply re-scale.

The energy and the enstrophy, namely $E(u)=\frac{1}{2}\int_{\T} |u|^2dx$ and $S(u)=\frac{1}{2}\int_{\T} |\mbox{curl} u|^2dx$, are conserved by the Euler velocity. In terms of the stream function $\phi$ we have
$$
E(\phi)=-\frac{1}{2}\int_{\T} \phi \D\phi dx
$$
and
$$
S(\phi)=\frac{1}{2}\int_{\T} |\D \phi|^2dx.
$$

We denote by $\{e^L_k\}_{k \in\Z^2}$ the following  orthonormal basis of $L^2(\T)$, 
$$
e^L_k=\frac{1}{L}e^{i\frac{2\pi}{L}k\cdot x},  \quad \forall ~k \in\Z^2.
$$
For all $u\in L^2(\T)$ we have 
$$
u(x,t)=\sum_{k>0} u_k^L(t)e_k^L(x),
$$
and we can identify the Sobolev spaces $H^{\beta} (\T)$ defined by
$$
H^{\beta}(\T):=\{ u:\T \to \R : ~(I-\Delta)^{\beta/2}u \in L^2(\T)\}
$$
with
\begin{equation}
H^{\beta}:=\left\{u=\sum_{k>0} u_k^L e_k^L : \sum_{k>0} \left(\frac{2\pi k}{L}\right)^{2\beta} |u_k^L|^2 < + \infty \right\}.
\end{equation}
We say that $k=(k_1,k_2)\in\Z^2$ is positive if $k_1>0$ or $k_1=0$ and $k_2>0$ and by $k^2$ we denote the inner product $k\cdot k =k_1^2+k_2^2$.

For all $\beta$, $H^\beta$ is a Hilbert space with inner product given by
$$
<u,v>_{\beta}:=\sum_{k>0} \left(\frac{2\pi k}{L}\right)^{2\beta} u_k^L \bar v_k^L.
$$

By means of the basis expansion on $L^2([0,L]^2)$, for $\phi^L(x,t)= \sum_{k>0} \phi_k^L(t)e_k^L(x)$, the equations reduce to an infinite dimensional system of first order ODEs
\begin{equation}
\frac{d}{dt}\phi_k^L(t)=B_k^L(\phi^L), \qquad \forall ~k\in \Z^2
\end{equation}
where
\begin{equation}\label{vectfield1}
B_L(\phi^L):=\sum_{k>0} B_k^L(\phi^L) e_k^L(x)
\end{equation}
and 
\begin{equation}\label{vectfield2}
B_k^L(\phi^L)=\frac{1}{L}\left( \frac{2\pi}{L} \right)^2 \sum_{\substack{h>0 \\ h\neq k}}\left[\frac{(h^\perp\cdot k)(k\cdot h)}{k^2} - \frac{h^\perp\cdot k}{2}\right]\phi_h^L\phi_{k-h}^L,
\end{equation}
where $h^\perp=(-h_2,h_1)$. We write $B_k^L(\phi^L)= \sum_h \alpha_{h,k}^L\phi_h^L\phi_{k-h}^L$, with 
\begin{equation}
\alpha_{h,k}^L=\frac{1}{L}\left( \frac{2\pi}{L} \right)^2 \left[\frac{(h^\perp\cdot k)(k\cdot h)}{k^2} - \frac{h^\perp\cdot k}{2}\right].
\end{equation}


\subsection{Notations}\label{notations}

Let us consider some relevant function spaces that will be used below. For all $\beta \in \R$ we define the local Sobolev spaces $ H_{loc}^{\beta}(\R^2)$ by
$$
H_{loc}^{\beta}(\R^2):=\{ u : \forall K\subset \R^2 \mbox{ compact}, ~ (I-\Delta)^{\beta/2}u \in L^2(K)\}.
$$
For negative or non-integer values of $\beta$, the operator $(I-\Delta)^{\beta/2}$ is considered as a pseudo-differential operator.
We may assume that the compact sets $K$ are of the type $K=[0,L]\times[0,L]$ for $L\in\N^*$. The spaces $H_{loc}^{\beta}(\R^2)$ are not normed spaces, however it is possible to equip them with the topology induced by the distances $\tilde d_{\beta,2}$ defined by 
\begin{equation}\label{tilde_d_beta,2}
\tilde d_{\beta,2}(u,v):=\sum_{L\in\N^*}2^{-L}\frac{ \| (I-\Delta)^{\beta/2}(u-v)\|_{L^2([0,L]^2)}}{1+ \| (I-\Delta)^{\beta/2}(u-v)\|_{L^2([0,L]^2)} }.
\end{equation}
In particular the metric spaces $\left(H_{loc}^{\beta}(\R^2); \tilde d_{\beta,2}\right)$ are complete for all $\beta\in \R$. For further results concerning local Sobolev spaces we refer to \cite{LSS}. Analogously, for all $\beta \in \R$ we define the spaces $W^{\beta,\infty}_{loc}(\R^2)$ by
$$
W^{\beta,\infty}_{loc}(\R^2):=\{ u : \forall K\subset \R^2 \mbox{ compact}, ~(I-\Delta)^{\beta/2}u \in L^{\infty}(K)\}.
$$
The metric spaces $W^{\beta,\infty}_{loc}(\R^2)$ are complete if endowed with the distances $ \tilde d_{\beta, \infty}$ defined by
\begin{equation}\label{tilde_d_beta,infty}
 \tilde d_{\beta, \infty}(u,v):=\sum_{L\in\N^*}2^{-L}\frac{ \| (I-\Delta)^{\beta/2}(u-v)\|_{L^{\infty}([0,L]^2)}}{1+ \| (I-\Delta)^{\beta/2}(u-v)\|_{L^{\infty}([0,L]^2)} }.
\end{equation}
For each fixed $\beta$ we have 
$$
W^{\beta,\infty}_{loc}(\R^2) \subseteq H_{loc}^{\beta}(\R^2).
$$

For each $L\in \N^*$, the norm $\| (I-\Delta)^{\beta/2} u\|_{L^{p}([0,L]^2)}$ is equivalent to the norm $\| D^\beta u\|_{L^{p}([0,L]^2)}$ for every $\beta \in \R$ and $1\leq p\leq +\infty$, thus it is possible to define other distances $ d_{\beta,2}$ and $ d_{\beta,\infty}$ such that $(H^{\beta}_{loc}(\R^2), d_{\beta,2})$ and $(W^{\beta, \infty}_{loc}(\R^2),  d_{\beta,\infty})$ are still complete. Indeed we have
\begin{align}\label{d_beta,2}
 \tilde d_{\beta,2}(u,v) & =\sum_{L\in\N^*}2^{-L}\frac{ \| (I-\Delta)^{\beta/2}(u-v)\|_{L^2([0,L]^2)}}{1+ \| (I-\Delta)^{\beta/2}(u-v)\|_{L^2([0,L]^2)} } \nonumber \\
& \leq \sum_{L\in\N^*} 2^{-L}C(L) \frac{ \| D^\beta( u-v)\|_{L^2([0,L]^2)} }{ 1+ \| D^\beta( u-v)\|_{L^2([0,L]^2)} } =:  d_{\beta,2} (u,v)
\end{align}
and 
\begin{equation}\label{d_beta,infty}
 \tilde d_{\beta,\infty}(u,v) \leq \sum_{L\in\N^*} 2^{-L}C_{\infty}(L) \frac{ \| D^\beta( u-v)\|_{L^{\infty}([0,L]^2)} }{ 1+ \| D^\beta( u-v)\|_{L^{\infty}([0,L]^2)} } =:  d_{\beta,\infty} (u,v),
\end{equation}
where $C(L)$ and $C_{\infty}(L)$ are constants at most proportional to $L^{\eta}$ for some $\eta\in \R$.

We say that a function $u$ belongs to the weighted Sobolev space $W^{\beta, \infty}(\R^2, 1+ |x|)$ for some fixed $\beta\in \R$ if 
$$
 \| (1+ |x|)^{-1}D^\beta u \|_{L^\infty(\R^2)} < +\infty.
$$ 
Whenever $\beta$ is negative or is not an integer the operator $D^\beta$ is understood as a pseudo-differential operator. The following inclusion holds: 
\begin{equation}\label{inclusion}
W^{\beta, \infty}(\R^2, 1+ |x|)\subseteq W^{\beta, \infty}_{loc}(\R^2).
\end{equation}

Below, we use $X \lesssim Y$ to denote the estimate $X \leq CY$ for some constant $C$. Unless stated otherwise $C$ is an unessential constant, in particular independent from the period $L$.


\section{The invariant measures}\label{inv_meas}

In the periodic setting and for each parameter $\gamma\in \R^+$, invariant probability measures, $\mu_{L,\gamma}$, were constructed, see \cite{AC}. In this section we define measures $\mu_{\gamma}$  as the weak limits of $\mu_{L,\gamma}$ when $L$ tends to infinity. Moreover we show that the support of $\mu_{\gamma}$ is the Sobolev space $H^\beta_{loc}(\R^2)$ for $\beta<1$. 

\subsection{Approximations of $\mu_{\gamma}$} 

On a probability space $(\Omega, \mathcal{F}, \P)$ consider a sequence of complex-valued i.i.d. Brownian motions, $\{W_{k^2}\}_{k\in \Z^2}$, and its increments, say them $\chi_k$, given by 
$$
\chi_k(\omega) = W_{k^2+1}(\omega)-W_{k^2}(\omega).
$$
Also, for each $L>0$ and $R=(R_1,R_2)\in \N^2$, consider the stochastic process defined
$$
\Phi_{L,R}(\omega,x):=\sum_{\substack{k>0 \\ k_1<LR_1\\ k_2<LR_2}}a_k^L(\omega)e^L_k(x),
$$
where 
$$
a_k^L(\omega):=\chi_k(\omega)\sqrt{\frac{2}{\gamma}}\left(\frac{L}{2\pi k}\right)^2 
$$
denotes for all fixed $k$ a complex-valued random variable with mean zero and variance $\frac{2}{\gamma}\left(\frac{L}{2\pi k}\right)^4$. Therefore $\Phi_{L,R}$ is a Gaussian vector with law and covariance matrix given respectively by,
$$
(\det M(L))^{-1/2} e^{-<a,M(L)^{-1}a>}\prod_{\substack{k>0 \\ k_1<LR_1\\ k_2<LR_2}}\gamma\frac{da_k^L(\omega)}{2\pi}
$$
and
$$
M(L)_{k,j}=\mathbb{E}_\P(a_k^L\bar a_j^L)=\delta^k_j\frac{2}{\gamma}\left(\frac{L}{2\pi k}\right)^4 ,
$$
where $\delta^k_j$ is the Kronecker symbol; thus we have,
$$
<a,M(L)^{-1}a>=\sum_{\substack{k>0 \\ k_1<LR_1\\ k_2<LR_2}}\left( \frac{2}{\gamma}\left(\frac{L}{2\pi k}\right)^4 \right)^{-1}|a_k^L(\omega)|^2 .
$$
Remark that, if $$\phi^{L,R}(x)=\sum_{\substack{k>0 \\ k_1<LR_1\\ k_2<LR_2}}\phi_k^{L,R} e^L_k(x),$$ then 
$$
\sum_{\substack{k>0 \\ k_1<LR_1\\ k_2<LR_2}}\left( \frac{2}{\gamma}\left(\frac{L}{2\pi k}\right)^4 \right)^{-1}|\phi_k^{L,R}|^2=\frac{\gamma}{2}\int_{\T} |\Delta \phi^{L,R}|^2dx;
$$
that is $$<\phi^{L,R},M(L)^{-1}\phi^{L,R}>=S(\phi^{L,R}),$$ where by $S(\phi^{L,R})$ we denoted the enstrophy. Hence the measure $d\mu_{L,\gamma}$, formally defined by
\begin{equation}\label{measureL}
d\mu_{L,\gamma}(\phi^L) := e^{-\frac{\gamma}{2}\int_{\T} |\Delta \phi^L|^2dx} \mathcal{D}\phi^L ,\quad \mathcal{D}\phi^L=
\prod_{k>0} \gamma\left(\frac{2\pi k}{L}\right)^4 \frac{d\phi_k^L}{2\pi}
\end{equation}
is the law of $\Phi_L$ on some Banach space, where 
\begin{equation}\label{PhiL}
\Phi_L(\omega, x):= \sum_{k>0} a_k^L(\omega)e^L_k(x).
\end{equation}
The measure $\mu_{L,\gamma}$ coincides with the Gibbs-type measure, relative to the enstrophy, defined in \cite{AC}. It was proved in \cite{AC} that $(H^\beta, H^2, \mu_{L,\gamma})$ is a complex abstract Wiener space for $\beta<1$; that is $H^2$ is a densely embedded Hilbert subspace of the Banach space $H^\beta$ and $\mu_{L,\gamma}$ is a Gaussian measure since
$$
\int e^{i\gamma l(\phi^L)} d\mu_{L,\gamma}(\phi^L)=e^{-\frac{1}{2}\gamma\| l\|_{2}^2}, \qquad \forall ~l\in (H^\beta)'\subset H^2.
$$
The space $H^\beta$ denotes the support of $\mu_{L,\gamma}$ and $H^2$ the associated Cameron-Martin space.

Below, we define $\Phi$ as the limit in $L^2(\Omega; H^{\beta}_{loc}(\R^2))$ of the sequence of random variables $\{\Phi_L\}_{L\in \N^*}$ given in equation \eqref{PhiL} and we define the measure $\mu_{\gamma}$ on functions of $\R^2$ as the image measure under the random variable $\Phi$. We follow the ideas of \cite{ASSuzz} where the Klein-Gordon equation on the real line is considered.

\begin{prop}\label{convergence}
The sequence $\{\Phi_L\}_{L\in \N^*}$ is a Cauchy sequence in $L^2(\Omega; H^{\beta}_{loc}(\R^2))$ for $\beta<1$.
\end{prop}
\begin{proof}
First observe that 
$$
W^{\beta, \infty}(\R^2) \subseteq W^{\beta, \infty}(\R^2, 1+ |x|) \subseteq W^{\beta, \infty}_{loc}(\R^2) \subseteq H^{\beta}_{loc}(\R^2)
$$
and that we can write for $0<L<S$
$$
\Phi_L-\Phi_S = \Phi_L - \Phi_{L,R} + \Phi_{L,R} -\Phi_{S,R} + \Phi_{S,R} -\Phi_{S}.
$$
We will show that $\mathbb{E}_\P \| D^\beta( \Phi_L - \Phi_{L,R}) \|^2_{L^{\infty}(\R^2)}$ converges to zero when $R$ tends to infinity uniformly in $L$ and that 
$\mathbb{E}_\P \|  (1+ |x|)^{-1}D^\beta( \Phi_{L,R} - \Phi_{S,R} )\|^2_{L^{\infty}(\R^2)}$ tends to zero when $L$ tends to infinity uniformly in $R$.
We have 
\begin{align*}
\mathbb{E}_\P  \| D^\beta(\Phi_L - \Phi_{L,R}) \|^2_{L^{\infty}(\R^2)} & = \mathbb{E}_\P \left[ \sup_{x\in \R^2}\left| \sum_{\substack{k_1\geq LR_1 \\ k_2\geq LR_2}} \left( \frac{2\pi k}{L}\right)^{\beta-2} \chi_k(\omega)\sqrt{\frac{2}{\gamma}}e_k^L(x)\right|  \right]^2 \\
& \leq \mathbb{E}_\P \left[   \frac{1}{L} \sum_{\substack{k_1\geq LR_1 \\ k_2\geq LR_2}} \left( \frac{2\pi k}{L}\right)^{\beta-2}\left| \chi_k(\omega)\right|\sqrt{\frac{2}{\gamma}}  \right]^2 \\
&=   \frac{1}{L^2} \frac{2}{\gamma}\sum_{\substack{k_1\geq LR_1 \\ k_2\geq LR_2}} \sum_{\substack{h_1\geq LR_1 \\ h_2\geq LR_2}}\left( \frac{2\pi k}{L}\right)^{\beta-2}  \left( \frac{2\pi h}{L}\right)^{\beta-2}  \mathbb{E}_\P [\chi_k(\omega)\bar \chi_h(\omega)] \\
&\leq \frac{2}{\gamma}  \sum_{\substack{k_1\geq LR_1 \\ k_2\geq LR_2}} \left( \frac{2\pi k}{L}\right)^{2\beta-4}\\
&  \lesssim \int_{[R, +\infty)^2}  \frac{dy }{y^{4-2\beta}} \leq \varepsilon 
\end{align*}
for $R$ sufficiently big and uniformly in $L$, since $\beta<1$. 

Now suppose that $L=2^n$ and $S=2^m$ with $n< m$; we have
\begin{align}
D^\beta (\Phi_{2^n, R} -\Phi_{2^m, R})= &\sqrt{\frac{2}{\gamma}}\Bigg[ \sum_{\substack{k>0 \\ k_1<2^{n}R_1 \\ k_2< 2^{n}R_2}}\left( \frac{2\pi k}{2^n}\right)^{\beta-2}\chi_k(\omega)2^{-n}e^{i\frac{2\pi}{2^n}k\cdot x} \label{first_series} \\
& - \sum_{\substack{l>0 \\ l_1<2^{m}R_1 \\ l_2< 2^{m}R_2}}\left( \frac{2\pi l}{2^m}\right)^{\beta-2} \chi_l(\omega)2^{-m}e^{i\frac{2\pi}{2^m}l\cdot x}\Bigg]. \label{second_series}
\end{align}
Also we have 
$$
\chi_l(\omega)2^{-m} \simeq W_{\frac{l^2+1}{2^{2m}}}(\omega)- W_{\frac{l^2}{2^{2m}}}(\omega)=:\varepsilon_{2^{-2m},l^2}(\omega),
$$
where here $\simeq$ denotes the symbol of identification in law, and where $\varepsilon_{2^{-2m},l^2}(\omega)$ can be written as
\begin{equation}\label{delta_equ}
\varepsilon_{2^{-2m},l^2}(\omega)\simeq \sum_{j=0}^{2^{n-m}-1} \varepsilon_{2^{-2n},(2^{n-m}l+ j)^2}(\omega).
\end{equation}
Indeed
\begin{align*}
 \sum_{j=0}^{2^{n-m}-1} \varepsilon_{2^{-2n},(2^{n-m}l+ j)^2}(\omega)& =  \sum_{j=0}^{2^{n-m}-1} W_{\frac{(2^{n-m}l+ j)^2+1}{2^{2n}}}(\omega)- W_{\frac{(2^{n-m}l+ j)^2}{2^{2n}}}(\omega) \\
 &\simeq W_{\frac{l^2+1}{2^{2m}}}(\omega)- W_{\frac{l^2}{2^{2m}}}(\omega) \\
 &= \varepsilon_{2^{-2m},l^2}(\omega).
\end{align*}

Therefore
\begin{align*}
D^\beta &(\Phi_{2^n, R} -\Phi_{2^m, R}) \simeq \\
&  \simeq  \sqrt{\frac{2}{\gamma}} \sum_{\substack{l>0 \\ l_1<2^{m}R_1 \\ l_2< 2^{m}R_2}} \sum_{j=0}^{2^{n-m}-1}  \varepsilon_{2^{-2n},(2^{n-m}l+ j)^2}(\omega) \left[ \frac{e^{i2\pi\frac{(2^{n-m}l+j)}{{2^n}}\cdot x}}{\left(\frac{ 2^{n-m}l+j}{2^n}\right)^{2-\beta}} - \frac{e^{i2\pi\frac{l}{2^m}\cdot x}}{\left(\frac{ l}{2^m}\right)^{2-\beta}}  \right].
\end{align*}
where we write $2^{n-m}l+ j:= (2^{n-m}l_1+ j; 2^{n-m}l_2+ j)$ for any $l=(l_1,l_2)\in \Z^2$ and $j\in \{0, \cdots, 2^{n-m}-1\}$. To get the last equality  (in law) we used: in \eqref{first_series} the change of variable $k= 2^{n-m}l + j$; and in \eqref{second_series} the replacement of \eqref{delta_equ}.

Take the $L^2(\Omega)$ norm of $D^\beta (\Phi_{2^n, R} -\Phi_{2^m, R})$:
$$
\mathbb{E}_\P |D^\beta (\Phi_{2^n, R} -\Phi_{2^m, R})|^2\lesssim\sum_{\substack{l>0 \\ l_1<2^{m}R_1 \\ l_2< 2^{m}R_2}} \sum_{j=0}^{2^{n-m}-1} 2^{-2n}\left[ \frac{e^{i2\pi\frac{(2^{n-m}l+j)}{{2^n}}\cdot x}}{\left(\frac{ 2^{n-m}l+j}{2^n}\right)^{2-\beta}} - \frac{e^{i2\pi\frac{l}{2^m}\cdot x}}{\left(\frac{ l}{2^m}\right)^{2-\beta}}  \right]^2
$$
and use that the directional derivatives of the function $y \in \R^2 \mapsto \frac{e^{i 2\pi y\cdot x}}{y^{2-\beta}}$ are bounded by $C(\beta)\frac{(1+2\pi|x|)}{y^{2-\beta}}$ in order to obtain
\begin{align*}
\sum_{\substack{l>0 \\ l_1<2^{m}R_1 \\ l_2< 2^{m}R_2}}  \sum_{j=0}^{2^{n-m}-1} & 2^{-2n}  \left[\frac{e^{i2\pi\frac{(2^{n-m}l+j)}{{2^n}}\cdot x}}{\left(\frac{ 2^{n-m}l+j}{2^n}\right)^{2-\beta}} - \frac{e^{i2\pi\frac{l}{2^m}\cdot x} }{ \left( \frac{ l}{2^m } \right)^{2-\beta}}  \right]^2 \\
& \lesssim \sum_{\substack{l>0 \\ l_1<2^{m}R_1 \\ l_2< 2^{m}R_2}} \sum_{j=0}^{2^{n-m}-1} 2^{-2n} \frac{(1+2\pi|x|)^2}{\left(\frac{l}{2^m}\right)^{4-2\beta}}\left( \frac{j}{2^{n}}\right)^2.
\end{align*}
Use the inequality
$$
\sum_{j=0}^{2^{n-m}-1} \left( \frac{j}{2^{n}}\right)^2\leq \sum_{j=0}^{2^{n-m}} 2^{-2m}=2^{n-3m}
$$
to get 
$$
\sum_{\substack{l>0 \\ l_1<2^{m}R_1 \\ l_2< 2^{m}R_2}} 2^{n-3m} \frac{(1+2\pi|x|)^2}{\left(\frac{l}{2^m}\right)^{4-2\beta}}  
\lesssim \varepsilon (1+|x|)^2 \int_{[a, +\infty)^2} \frac{dy}{y^{4-2\beta}}  
\lesssim \varepsilon (1+|x|)^2 
$$
for $m$ sufficiently big and uniformly in $R$ since $\beta<1$ and $a\in (0,1)$. Back to $L$ and $S$ we have $\mathbb{E}_\P\|  (1+ |x|)^{-1}D^\beta( \Phi_{L,R} - \Phi_{S,R} )\|^2_{L^{\infty}(\R^2)}\leq \varepsilon$ for $L$ sufficiently big and uniformly in $R$. 
\end{proof}

In the following we denote by $\mu_{\gamma}$ the law of $\Phi$ where $\Phi$ is the limit of $\{\Phi_L\}_{L\in\N^*}$ in $L^2(\Omega; H^{\beta}_{loc}(\R^2))$. This $L^2$-convergence implies that $\mu_{L,\gamma}$ converges weakly to $\mu_{\gamma}$ in $H^{\beta}_{loc}(\R^2)$ when $L$ tends to infinity.

\subsection{Support of $\mu_{\gamma}$}

Here we study the support of the measure $\mu_{\gamma}$. Since $\mu_{\gamma}$ is the law of $\Phi$, its support is defined as the space in which $\Phi(\omega, \cdot)$ takes values $\P$-almost surely.

\begin{prop}
Let $\beta<1$, we have
$$
\supp(\mu_{\gamma}) = H^{\beta}_{loc}(\R^2).
$$
\end{prop}
\begin{proof}
We have
$$
\mathbb{E}_\P  d_{\beta,2} (\Phi,0)\leq \mathbb{E}_\P  d_{\beta,2} (\Phi,\Phi_{L,R}) + \mathbb{E}_\P  d_{\beta,2} (\Phi_{L,R},0),
$$
where $d_{\beta,2}$ denotes the metric for $H^{\beta}_{loc}(\R^2)$ defined in \eqref{d_beta,2}.
On one hand and by Proposition \ref{convergence}, $\mathbb{E}_\P d_{\beta,2} (\Phi,\Phi_{L,R})$ tends to zero when $L$ and $R$ tend to infinity. On the other 
$\mathbb{E}_\P d_{\beta,2} (\Phi_{L,R},0) \leq C < +\infty$ since we have
\begin{align*}
\mathbb{E}_\P d_{\beta,2} (\Phi_{L,R},0) & \leq   \sum_L 2^{-L}C(L)\mathbb{E}_\P \|D^\beta \Phi_{L,R}\|_{L^2([0,L]^2)} \\
& \lesssim \sum_L 2^{-L}C(L) < +\infty.
\end{align*}
We used the fact that 
\begin{align*}
\mathbb{E}_\P\| D^\beta \Phi_{L,R}\|^2_{L^2([0,L]^2)}& = \sum_{ \substack{k>0 \\ k_1<LR_1 \\ k_2<LR_2}} \left( \frac{2\pi k}{L}\right)^{2\beta} \mathbb{E}_\P |a_k^L(\omega)|^2\\
&\lesssim \sum_{ \substack{k>0 \\ k_1<LR_1 \\ k_2<LR_2}} \left( \frac{ k}{L}\right)^{2\beta-4} \\
& \lesssim \int_{[a,+\infty)^2} \frac{dy}{y^{4-2\beta}} \leq C < +\infty
\end{align*}
for $a>0$ small enough; and that $C(L)$ depends on the period as previously explained in Subsection \ref{notations}.
\end{proof}

Formally the measure $\mu_{\gamma}$ is given by
\begin{equation}\label{measureMU}
d\mu_{\gamma} (\phi)= \frac{1}{Z} e ^{-\frac{\gamma}{2}\int_{\R^2} |\Delta \phi |^2dx} \mathcal{D}\phi
\end{equation}
where $Z$ is a suitable renormalizing constant. For all fixed $L\in \N^*$, the measure $\mu_{\gamma}$ on functions restricted to the compact phase space $[0,L]^2$ is in fact the measure $\mu_{L,\gamma}$. As in \cite{AC} for $(H^\beta, H^2, \mu_{L,\gamma})$ we can show that $(H^\beta_{loc}(\R^2), H^2_{loc}(\R^2), \mu_{\gamma})$ is a complex abstract Wiener space for $\beta<1$.


\section{The velocity flow on $\R^2$}\label{muFlow}

The aim of this section is to prove global existence and uniqueness of the Euler flow on the plane, under which $\mu_{\gamma}$ is invariant.

\subsection{Approximations of the vector field $B$}

We start by recalling some properties of the vector field $B_L$ in the periodic setting, given by equations \eqref{vectfield1}-\eqref{vectfield2} and previously derived in \cite{AC}. 

\begin{prop}\label{divergence_free}
The vector field $B_L$ is divergence-free with respect to the measure $\mu_{L,\gamma}$, that is $\delta_{\mu_{L,\gamma}} B_L=0$.
\end{prop}
\begin{proof}
We refer to \cite{AC} and only remark that the conservation of the enstrophy is essential to prove the statement.
\end{proof}

We recall the proof of the $L^p_{\mu_{L,\gamma}}$-regularity of $B_L$ for any $p\geq 1$, as we are interested in the dependence on the period $L$ of such estimates. For further details see \cite{AC} or \cite{C}.

\begin{prop}\label{p_reg}
Let $\beta<-1$, then the vector field $B_L\in L^p_{\mu_{L,\gamma}}(H^\beta; H^\beta)$ for all $p\geq 1$.
\end{prop}
\begin{proof}
It is enough to show that $\Em\|B_L(\varphi^L)\|^{2p}_{H^\beta}<+\infty$ for all $p>1$. We have
\begin{align*}
\Em\|B_L(\varphi^L)\|^{2p}_{H^\beta}&=\Em  \left[\sum_{k>0} \left( \frac{2\pi k}{L}\right)^{2\beta}|B_k^L(\varphi^L)|^2 \right]^p \\
& \leq \left[ \sum_{k>0}\left(  \frac{2\pi k}{L}\right)^{2\beta} \left(\Em |B_k^L(\varphi^L)|^{2p} \right)^{1/p} \right]^p
\end{align*}
From $B_{k}^L(\varphi^L)=\sum_h \alpha^L_{h,k}\phi_h^L \phi_{k-h}^L$ we have that
\begin{align*}
\Em|B_k^L(\phi^L)|^{2p}& = \left[  \sum_{h,h'}\alpha_{h,k}^L\alpha_{h',k}^L\Em(\phi_h^L\phi_{k-h}^L\bar \phi_{h'}^L\bar \phi_{k-h'}^L) \right]^p\\ 
& \leq \left[  \sum_{h,h'}\alpha_{h,k}^L\alpha_{h',k}^L \left( \Em (\phi_h^L\phi_{k-h}^L\bar \phi_{h'}^L\bar \phi_{k-h'}^L)^p\right)^{1/p}  \right]^p \\
&= \left[  2 \sum_{h}|\alpha_{h,k}^L|^2 \left(\Em|\phi_h^L|^{2p}\right)^{1/p}\left(\Em|\phi_{k-h}^L|^{2p}\right)^{1/p}  \right]^p\\
& \lesssim p!^2 \left[ \sum_{h}|\alpha_{h,k}^L|^2\frac{L^8}{h^4(k-h)^4} \right]^p \\
&\leq p!^2 \left[ L^2 \sum_{h}\left[\frac{(h^\perp\cdot k)(k\cdot h)}{k^2} - \frac{h^\perp\cdot k}{2}\right]^2\frac{1}{h^4(k-h)^4} \right]^p \\
& \leq (L^{2} C)^p <+\infty , \qquad \forall ~p>1. 
\end{align*}
Therefore, since $\beta<-1$,
\begin{equation}\label{B}
\Em\|B_L(\varphi^L)\|^{2p}_{H^\beta} \lesssim \left( \frac{1}{L^{2\beta-2}}\sum_{k>0}\frac{1}{k^{-2\beta}} \right)^p \leq (L^{2-2\beta} C)^p <+\infty, \qquad \forall ~p>1.
\end{equation}
\end{proof}

\begin{rem}\label{conv B_L to B}
For the vector field on $[0,L]^2$ the expression $B_L(\phi)=\sum_k B_k^L(\phi) e_k^L(x)$ where $B_k^L$ is defined in \eqref{vectfield2} is valid. Note however that the Euler vector field does not depend on $L$; it is the same on every finite phase space approximation and thus $B_L$ trivially converges to $B$, the Euler vector field on $\R^2$, when $L$ goes to infinity. 
\end{rem}

Next we show that $B: H^\beta_{loc}(\R^2) \to H^\beta_{loc}(\R^2)$ is regular with respect to $L^p_{\mu_{\gamma}}$ for all $p\geq 1$. 

\begin{cor}\label{reg}
Let $\beta<-1$, then $B\in L^p_{\mu_{\gamma}}(H^\beta_{loc}(\R^2); H^\beta_{loc}(\R^2))$ for all $p\geq 1$.
\end{cor}
\begin{proof}
We show that $\E_{\mu_{\gamma}} \left| d_{\beta,2}(B(\phi),0) \right|^{2p}< +\infty$ for all $p>1$, where $ d_{\beta,2}$ denotes the metric for $H^{\beta}_{loc}(\R^2)$ defined in \eqref{d_beta,2}. We have
\begin{align*}
\E_{\mu_{\gamma}} \left| d_{\beta,2}(B(\phi),0)  \right|^{2p}& = \E_{\mu_{\gamma}} \left| \sum_{L\in \N^*} 2^{-L}C(L) \frac{\|B(\phi)\|_{H^\beta([0,L]^2)}}{1+ \|B(\phi)\|_{H^\beta([0,L]^2)} } \right|^{2p} \\
& \leq \left[  \sum_{L\in \N^*} 2^{-L}C(L)  \left( \E_{\mu_{\gamma}}  \frac{\|B(\phi)\|^{2p}_{H^\beta([0,L]^2)}}{(1+ \|B(\phi)\|_{H^\beta([0,L]^2)} )^{2p}} \right)^{1/2p}  \right]^{2p} \\
& \leq \left[ \sum_{L\in \N^*} 2^{-L}C(L)\left(\E_{\mu_{L,\gamma}} \| B_L(\phi)\|^{2p}_{H^\beta([0,L]^2)}\right)^{1/2p} \right]^{2p},
\end{align*}
where we got the last inequality from Proposition \ref{p_reg}. Again, from estimative \eqref{B} and since $\beta<-1$, we conclude 
$$
\E_{\mu_{\gamma}} \left|d_{\beta,2}(B(\phi),0)  \right|^{2p} \lesssim \left[ \sum_{L\in \N^*} 2^{-L}C(L) L^{2-2\beta} \right]^{2p} <+\infty, \qquad \forall ~p>1.
$$
\end{proof}

In the next Lemma, we prove existence for the approximated Euler equations.

\begin{lem}
For any fixed $L\in\N^*$ and $R\in \N^2$ we consider a phase space projection on $[0,L]^2$ and a finite dimensional approximation of equation \eqref{vorEuler}; thus there exists a globally defined Euler flow, say it $U^{L,R}$, defined on $H^\beta_{loc}(\R^2)$. 
\end{lem}
\begin{proof}
We study the following system of ODEs for all $k\in \Z^2$ with $k>0$, $k_1<LR_1$ and $k_2<LR_2$:
\begin{align*}\label{LR}
\frac{d}{dt}U_k^{L,R}(t,\phi^{L,R})& =B^{L,R}_k(U^{L,R}(t,\phi^{L,R}))\\
U_k^{L,R}(0,\phi^{L,R}) & =\phi^{L,R}_k
\end{align*}
for 
$$
\phi^{L,R}(t,x)=\sum_{ \substack{k>0 \\ k_1<LR_1 \\ k_2<LR_2}}\phi^{L,R}_k(t) e_k^L(x) \in \C^d,
$$
where $d=d(R):=\#\{k\in \Z^2 :  k>0 \mbox{ and } k_i<LR_i \mbox{ for } i=1,2 \}$ and where
$$
B^{L,R}_k(\phi^{L,R})= \frac{1}{L}\left( \frac{2\pi}{L} \right)^2 \sum_{\substack{h>0 \\ h\neq k \\ h_1<LR_1 \\ h_2<LR_2}}\left[\frac{(h^\perp\cdot k)(k\cdot h)}{k^2} - \frac{h^\perp\cdot k}{2}\right]\phi_h^{L,R}\phi_{k-h}^{L,R}.
$$
From the regularity of the finite dimensional quadratic vector field $B^{L,R}$  we know that there exists an associated global flow on $\C^d$, that is for all positive $k\in \Z^2$ with $k_1<LR_1$ and $k_2<LR_2$ we have
$$
U_k^{L,R}(t,\phi^{L,R})= \phi^{L,R}_k + \int_0^t B^{L,R}_k(U^{L,R}(s,\phi^{L,R}))ds, \qquad \forall ~t\in \R.
$$
Now, for $\phi^L\in H^\beta$ we write
$$
\phi^L= \Pi_R \phi^L + \Pi_R^\perp \phi^L= \phi^{L,R} + \Pi_R^\perp \phi^L,
$$
where $\Pi_R$ is the orthogonal projection on the subspace spanned by $\{e_k :  k>0 \mbox{ and } k_i<LR_i \mbox{ for } i=1,2\}$. Therefore, if we define
$$
U_k^{L,R}(t,\phi^L):= U_k^{L,R}(t,\phi^{L,R}) + \Pi_R^\perp \phi^L,
$$
then $U^{L,R}(t,\phi^L)$ is in fact a $B^{L,R}$-flow on $H^\beta([0,L]^2)$. 
Finally, for $\phi \in H^\beta_{loc}(\R^2)$ we write
$$
\phi = \left. \phi\right|_{{[0,L]^2}} + \left. \phi\right|_{{[0,L]^2}^C}=  \phi^L+ \left. \phi\right|_{{[0,L]^2}^C}
$$
and we define
$$
U_k^{L,R}(t,\phi):= U_k^{L,R}(t,\phi^L) + \left. \phi\right|_{{[0,L]^2}^C};
$$
it follows that $U^{L,R}(t,\phi)$ is in fact a $B^{L,R}$-flow on $H^\beta_{loc}(\R^2)$. From the conservation of the energy we know that the flow is defined for all times. Furthermore we have 
$$
U^{L,R}(t,\phi)=\sum_{ \substack{k>0 \\ k_1<LR_1 \\ k_2<LR_2}} U_k^{L,R}(t,\phi)e_k^L
$$
with $U_k^{L,R}(\cdot,\phi)\in C(\R; \C)$ for all $k$.
\end{proof}

\subsection{Existence of a unique invariant flow}

Here, we prove the existence of a unique and invariant flow for \eqref{vorEuler} taking values in $H^\beta_{loc}(\R^2)$ for $\beta<-1$.

\begin{thm}\label{existence}

Let $\beta<-1$. There exists a globally defined flow $ U(\cdot, \phi)\in C(\R;  H^\beta_{loc}(\R^2) )$ for $\mu_\gamma$- a.e. $ \phi \in H^\beta_{loc}(\R^2)$, such that
\begin{enumerate}
\item \label{one} 
\begin{equation*}
U(t, \varphi)=\varphi+ \int_0^t B(U(s, \varphi))ds, \qquad \mu_{\gamma}-a.e. ~ \varphi \in H^\beta_{loc}(\R^2) , ~\forall ~t\in \R;
\end{equation*}
\item \label{two}
the flow is unique;
\item \label{three}
the measure $\mu_{\gamma}$ is invariant under the flow:
\begin{equation*}
\int f(U(t, \phi))d\mu_\gamma(\phi)=\int f(\varphi)d\mu_{\gamma}(\varphi), \qquad \forall f\in C_b, ~~ \forall~t\in \R.
\end{equation*}
\end{enumerate}
\end{thm}

\begin{proof}
~
\vspace{5mm}
\paragraph{\bf (i) Existence} 
Consider $U_k^{L,R}$ as a stochastic process with law on $C(\R; H^\beta_{loc}(\R^2))$. From Proposition \ref{convergence}, we know that $\mu_{L,\gamma}^{R}$ is a weakly convergent sequence of probability measures in $H^\beta_{loc}(\R^2)$. Therefore, by Skorohod's theorem there exists a probability space $(\tilde\Omega, \mathcal{\tilde F}, \tilde P)$ and two stochastic processes $\tilde U^{L,R}, \tilde U$ with laws respectively $\mu_{L,\gamma}^{R}, \mu_{\gamma}$, such that $\tilde U^{L,R}(t, \tilde w)$ converges to $ \tilde U(t, \tilde w)$ $\tilde P$- a.e. $\tilde w$ and for all $t \in \R$, when $L,R$ tend to infinity. In particular, it follows 
\begin{equation}\label{stat_inv_int}
\int f(\tilde U(t, \tilde\w))d\tilde P(\tilde\w)=\int f(\varphi)d\mu_{\gamma}(\varphi), \qquad \forall f\in C_b.
\end{equation} 
Moreover for all $L\in \N^*$ and for $\beta<-1$, we have
$$
\int \sum_k \left( \frac{2\pi k}{L} \right)^{2\beta} |\tilde U^L_k(t,\tilde w)|^2 d\tilde P(\tilde w) = \int \|\varphi^L\|_{H^\beta}^2 d\mu_{L,\gamma}(\varphi^L) \leq C <+\infty,
$$
this implies that $\tilde U(t,\tilde w)$ takes values in $H^\beta_{loc}(\R^2)$ for all $t\in \R$.

Now, to prove the following:
\begin{equation}\label{stat_ex_int}
\tilde U(t,\tilde \w)=\tilde U(0, \tilde\w) + \int_0^t B(\tilde U(s, \tilde\w))ds, \qquad \tilde P-a.e.\,\tilde\w,~ \forall ~t\in \R, 
\end{equation}
 we have to check that 
$$
\mathbb{E}_{\tilde P} d_{\beta,2} (  \int_0^t [B_k^{L,R}(\tilde U^{L,R}(s,\tilde\w))- B_k(\tilde U(s,\tilde\w))]ds; 0 ) $$
tends to $0$ when $L$ and $R$ tend to infinity. We have
\begin{align*}
\mathbb{E}_{\tilde P} d_{\beta,2} (  \int_0^t &[B_k^{L,R}(\tilde U^{L,R}(s,\tilde\w)) - B_k(\tilde U(s,\tilde\w))]ds; 0 ) \leq \\
& \mathbb{E}_{\tilde P} d_{\beta,2} (  \int_0^t [B_k^{L,R}(\tilde U^{L,R}(s,\tilde\w))- B_k(\tilde U^{L,R}(s,\tilde\w))]ds; 0 ) \\
& +  \mathbb{E}_{\tilde P} d_{\beta,2} (  \int_0^t [B_k(\tilde U^{L,R}(s,\tilde\w))- B_k(\tilde U(s,\tilde\w))]ds; 0 ) .
\end{align*}
The first term is bounded by
$$
\sum_{L\in\N^*} 2^{-L}C(L) \sum_k \left( \frac{2\pi k}{L} \right)^{2\beta} \int_0^t \mathbb{E}_{\tilde P} |B_k^{L,R}(\tilde U^{L,R}(s,\tilde\w))- B_k(\tilde U^{L,R}(s,\tilde\w))|^2 ds.
$$
It converges to $0$ when $L$ and $R$ tend to infinity by the invariance of the measure  and the $L^2$ convergence of $B_k^{L,R}$ towards $B_k$.
 Analogously the second term is bounded by
$$
\sum_{L\in\N^*} 2^{-L} C(L) \sum_k  \left( \frac{2\pi k}{L} \right)^{2\beta} \int_0^t \mathbb{E}_{\tilde P} |B_k(\tilde U^{L,R}(s,\tilde\w))- B_k(\tilde U(s,\tilde\w))|^2 ds.
$$
This term also converges to $0$ when $L$ and $R$ go to infinity by the equi-integrability of the functions $B_k(\tilde U^{L,R}(s,\tilde\w))$ and the convergence of the flows $\tilde U^{L,R}(s,\tilde\w)$ towards $\tilde U(s,\tilde\w)$ (similar to the arguments used in \cite{AC}).

Up to now we only proved the intermediary existence result \eqref{stat_ex_int}; we will finally get statement \ref{one} after the proof of uniqueness, see equation \eqref{stat_ex}.


\vspace{5mm}
\paragraph{\bf (ii) Uniqueness}
Every time that we consider the vorticity equation projected on the torus a uniqueness argument, similar to the one presented in \cite{AF}, applies. Uniqueness of the velocity flow follows from uniqueness of its law seen as the solution of the corresponding continuity equation; as in the classical DiPerna Lions approach for vector fields with low regularity, see \cite{dPL}. 
We use the machinery from \cite{AF}, namely Theorem 4.7, to say that the law of $\tilde U^L$ is a Dirac measure on the trajectories, the proof of this relies on the fact that the solution of the continuity equation is unique. 

Now, let $k_t^L$ be the Radon-Nikodym density of $d(\tilde U^L(t, \cdot)* \tilde P)$ with respect to $d\mu_{L,\gamma}$ at time $t\in\R$. We have that $k_t^L$ is a bounded weak solution of
\begin{align}\label{cont_eq_infty}
\frac{d}{dt} k_t^L(\phi) & = - <B_L(\phi), \nabla k_t^L (\phi) >_{\beta}, \qquad \mbox{ in } \R^+\times H^\beta; \\
k_0^L(\phi)& =1; \nonumber
\end{align}
that is
\begin{equation}\label{sol_cont_eq}
\int_0^\infty \int_{H^\beta} k_t^L(\phi) \left(-\partial_t f+ <B_L(\phi), \nabla f  >_{\beta} \right)d\mu_{L,\gamma}(\phi)dt =  \int_{H^\beta} f(0, \phi)d\mu_{L,\gamma}(\phi), \quad \forall f\in \mathcal{D}_t,
\end{equation}
where $\mathcal{D}_t$ denotes the space of differentiable functions on $\R^+\times H^\beta$ depending on a finite number of coordinates. Clearly $k_t^L\equiv 1$ is a solution of \eqref{cont_eq_infty}, below we show that it is unique.
We remark that for each Galerkin approximation of $B_L$, $B^{n}_L$ with $n\in \N$, uniqueness holds since $B^{n}_L$ is quadratic. Thus $k_t^{L,n}\equiv 1$ is the unique solution of the truncated continuity equation.
Now, let $\tilde k_t^L$ be another weak solution of \eqref{cont_eq_infty}, that is $\tilde k_t^L$ verifies \eqref{sol_cont_eq}.
We have
\begin{align*}
&\int_0^\infty \int_{H^\beta} \tilde k_t^L(\phi) \left(-\partial_t f+ <B_L(\phi), \nabla f  >_{\beta} \right)d\mu_{L,\gamma}(\phi)dt -  \int_{H^\beta} f(0, \phi)d\mu_{L,\gamma}(\phi) \\
&= \int_0^\infty \int_{H^\beta} \tilde k_t^L(\phi^n) \left(-\partial_t f+ <B_L(\phi^n), \nabla f  >_{\beta} \right)d\mu_{L,\gamma}^n(\phi^n)dt \int_{H^\beta}d\mu_{L,\gamma}^{n, \perp}(\phi^{n,\perp})\\
& -  \int_{H^\beta} f(0, \phi^n)d\mu_{L,\gamma}^n(\phi^n)\int_{H^\beta}d\mu_{L,\gamma}^{n, \perp}(\phi^{n,\perp})\\
&= \int_0^\infty \int_{H^\beta}  \left(-\partial_t f+ <B_L(\phi), \nabla f  >_{\beta} \right)d\mu_{L,\gamma}(\phi)dt -  \int_{H^\beta} f(0, \phi)d\mu_{L,\gamma}(\phi) 
\end{align*}
where we used that $\tilde  k_t^L(\phi^n)= \tilde k_t^{L,n}(\phi)=1$ and $B_L(\phi^n)=B^{n}_L(\phi)$. From the arbitrariness of $f \in \mathcal{D}_t$ we conclude that $k_t^L\equiv 1$ is the unique solution of \eqref{cont_eq_infty} in $\R^+\times H^\beta$. To get the negative values of $t$ we repeat the same reasoning for the map $t\mapsto k_{-t}^L$. 

Therefore, by Theorem 4.7 from \cite{AF}, $\tilde U^L(t, \tilde \w)$ is unique in the sense that any other $B_L$-flow, $U'^L(t,\tilde\w)$, is such that 
$$
\tilde U^L(\cdot,\tilde \w)=U'^L(\cdot,\tilde \w), \qquad \tilde P- a.e. \, \tilde \w\in \tilde\Omega.
$$
Moreover, on each compact phase space, the law of the Euler flow is a Dirac measure on the trajectories, implying that the solution is in fact deterministic. That is, we have that
$$
U^L(t, \varphi^L)= \varphi^L+ \int_0^t B_L(U^L(s, \varphi^L))ds, \qquad \mu_{L,\gamma}-a.e. ~\varphi^L, \forall ~t\in \R
$$
is the unique $B_L$-flow. Now, if $M\in \N^*$ is such that $M>L$, from $\left. \varphi^L \equiv \varphi^M \right|_{[0,L]^2}$ and $\left. B_L(t, \varphi^L) \equiv B_M(t, \varphi^M\right|_{[0,L]^2})$ we get
$$
U^L(t, \varphi^L) \equiv U^M(t, \left. \varphi^M)\right|_{[0,L]^2}, \qquad \forall t\in \R.
$$
Therefore uniqueness holds for the velocity flow $\tilde U(t, \tilde w)$ defined in the previous theorem which is in fact deterministic; we denote it by
\begin{equation}\label{stat_ex}
U(t, \varphi)=\varphi+ \int_0^t B(U(s, \varphi))ds, \qquad \mu_{\gamma}-a.e. ~ \varphi \in H^\beta_{loc}(\R^2) , ~\forall t\in \R.
\end{equation}

\vspace{5mm}
\paragraph{\bf (iii) Invariance}
The measure $\mu_{\gamma}$ is invariant under the deterministic flow $U(t, \varphi)$ defined for $t\in \R$ and $\varphi \in H^\beta_{loc}(\R^2)$. Indeed for all $ f \in C_b$ we have 
$$
 \int f d\mu_{\gamma}= \lim_L \int f d\mu_{L,\gamma}=  \lim_L \int f( U(t, \varphi))d\mu_{L,\gamma}=  \int f( U(t, \varphi)) d\mu_{\gamma}, \qquad \forall t \in \R.
$$


It only remains us to prove that for every fixed initial data $\phi \in H^\beta_{loc}(\R^2)$, $U(\cdot, \phi)$ is a continuous function of time in $H^\beta_{loc}(\R^2)$. Let $t > t' \in \R$ be such that $|t-t'|<\delta$ for some $\delta>0$, from the invariance property and Proposition \ref{p_reg} we have
\begin{align*}
\mathbb{E}_{\mathbb{P}}  \sup_{|t-t'|<\delta} d_{\beta,2}(U(t, \phi); U(t',\phi))  & 
= \mathbb{E}_{\mathbb{P}}  \sup_{|t-t'|<\delta} \sum_L 2^{-L}C(L) \frac{\| \int_{t'}^t B(U(s,\phi))\|_{H^\beta([0,L]^2)}}{1+ \| \int_{t'}^t B(U(s,\phi))\|_{H^\beta([0,L]^2)}} \\
&\leq \delta \sum_L 2^{-L}C(L) \mathbb{E}_{\mathbb{P}} \|B (U(s, \phi))\|_{H^\beta([0,L]^2)} \\
&= \delta \sum_L 2^{-L}C(L) \mathbb{E}_{\mathbb{P}} \|B(\phi)\|_{H^\beta([0,L]^2)} \underset{\delta\to 0}\rightarrow 0.
\end{align*}

\end{proof}

\begin{rem}
The proof of the existence of a two-dimensional Euler flow partially relies on the ideas from \cite{AC}, by which it is possible to construct a probabilistic (in the sense of the proof above) flow on the plane. However, the result of Theorem \ref{existence} above is stronger, since, as a by product of the proof of uniqueness, we get that this probabilistic flow is in fact determinist. 
\end{rem}

\begin{rem}
With respect to \cite{AF}, we are in a very particular case: $B_L$ is autonomous, quadratic and divergence-free. The latter hypothesis permit to show uniqueness in a simpler way that the one presented in \cite{AF}, in particular we do not need any additional assumption on the gradient of $B_L$. Moreover, the vector field being autonomous, we are not in the case of Depauw's counterexample about non-uniqueness of weak solutions for the continuity equations, see \cite{Cri} for more details.
\end{rem}


\subsection{Continuity}

The flow is continuous from $H^\beta_{loc}(\R^2)$ to $H^\beta_{loc}(\R^2)$ on the support of $\mu_{\gamma}$ for all $ t\in \R$. We write
\begin{align*}
\mathbb{E}_{\mu_{\gamma}} d_{\beta,2}(U(t, \varphi_1); U(t, \varphi_2)) & \leq  \mathbb{E}_{\mu_{\gamma}} d_{\beta,2}(U(t, \varphi_1); U^n(t, \varphi_1))  \\
&+ \mathbb{E}_{\mu_{\gamma}} d_{\beta,2}(U^n(t, \varphi_1); U^n(t, \varphi_2))  \\
&+ \mathbb{E}_{\mu_{\gamma}} d_{\beta,2}(U^n(t, \varphi_2); U(t, \varphi_2))
\end{align*}
where $U^n$ denotes a finite dimensional approximation of $U$. On one hand there exist $n_1, n_2\in\N$ such that for every $n\geq \max\{n_1, n_2\}$ 
$$
\mathbb{E}_{\mu_{\gamma}} d_{\beta,2}(U(t, \varphi_1); U^n(t, \varphi_1)) \leq \frac{\varepsilon}{3}\quad  \mbox{ and } \quad \mathbb{E}_{\mu_{\gamma}} d_{\beta,2}(U^n(t, \varphi_2); U(t, \varphi_2))  \leq \frac{\varepsilon}{3}.
$$
On the other, for a fixed $n\geq \max\{n_1, n_2\}$, we have that $U^n$ is continuous; indeed it is the flow for the quadratic vector field $B^n$. Thus there exists a positive $\delta$ such that for 
$d_{\beta,2}(\varphi_1; \varphi_2)\leq \delta$ we have 
$$
\mathbb{E}_{\mu_{\gamma}} d_{\beta,2}(U^n(t, \varphi_1); U^n(t, \varphi_2))\leq \frac{\varepsilon}{3}.
$$

\subsection*{Acknowledgements}\noindent
The authors were partially supported by Portuguese FCT grant PTDC/MAT-STA/0975/2014. The second author was also funded by the LisMath fellowship PD/BD/52641/2014, FCT, Portugal.


\bibliography{biblio}
\bibliographystyle{siam}

\end{document}